\documentclass[12pt]{article}
\usepackage{amsthm,amsfonts,amsmath,amssymb}

\newcommand{\eq}{\begin{equation}}
\newcommand{\en}{\end{equation}}

\newcommand{\ex}{\mathbb E}

\newtheorem{theorem}{\large Theorem}[]
\newtheorem{proposition}[theorem] {\large Proposition}

\newtheorem{corollary}[theorem]{\large Corollary}

\begin{document}
\title{Moments of random sums and Robbins' problem of optimal stopping}
\author{Alexander Gnedin\thanks{Postal address:
 Department of Mathematics, Utrecht University,
 Postbus 80010, 3508 TA Utrecht, The Netherlands. E-mail address: A.V.Gnedin@uu.nl}
~~and~~ Alexander Iksanov\thanks{Postal address: Faculty of
Cybernetics, National T. Shevchenko University of Kiev,
Kiev-01033, Ukraine. E-mail address: iksan72@mail.ru}}
\date{}
\maketitle
\begin{abstract}
\noindent
Robbins' problem of optimal stopping asks one to minimise the expected {\it rank} of observation chosen  by some nonanticipating 
 stopping rule. We settle a conjecture regarding the {\it value} of the stopped variable under the rule optimal in the sense of the rank, 
by embedding the problem
in a much  more general context of selection problems 
with the nonanticipation constraint lifted, and
with the payoff growing like a power function of the rank.

\end{abstract}

\paragraph{1.} Let $X_1,\dots, X_n$ be independent random variables sampled sequentially from  the uniform $[0,1]$ distribution,
and let $Y_1<\ldots<Y_n$ be their order statistics.
The rank $R_j$ of the variable $X_j$ is defined by setting $R_j=k$ on the event $X_j=Y_k$.
Robbins' problem of optimal stopping \cite{BrSurv} asks one to minimize the expected rank
$\ex R_\tau$ over all stopping times $\tau$ that assume values in $\{1,\dots,n\}$
and are adapted to the natural filtration of the sequence $X_1,\dots, X_n$.
Let $\tau_n$ be the optimal stopping time.
The minimum expected   rank $\ex R_{\tau_n}$ increases as $n$ grows, and converges to some finite limit $v$ whose exact value is unknown.
The closest known upper bound is slightly less than $7/3$.  Finding $v$ or even improving the existing rough bounds remains a challenge.
A major source of difficulties is that the
optimal stopping time $\tau_n$ is a
 very complicated function of the sample. It seems that $\tau_n$ has not been computed for $n>3$.
Moreover, for large $n$ there is no simplification, and the complexity of the optimal stopping time persists in
the `$n=\infty$' limiting form of the problem \cite{GnRank}.

In a recent paper Bruss and Swan  \cite{BrSw} stressed that it is not even known if $\limsup_n n \ex X_{\tau_n}$ is finite.
They mentioned that the property was first conjectured in \cite{Assaf}.
While the conjecture stems from the attempts to bound $v$ by the comparison  with much simpler problem of minimising
$\ex X_\tau$ (or minor variations of the problem), it seems that the question is of independent interest
 as a relation between the stopped sample value  and its rank.
In this note we  settle the conjecture by proving a considerably more general assertion:
\begin{proposition}
Fix $p>0$.  For $n=1,2,\dots$ let $\sigma_n$ be a random variable 
with range $\{1,\ldots,n\}$ and arbitrary joint  distribution with $X_1,\dots,X_n$. Then
\begin{equation}\label{conj}
\limsup_n \ex [R_{\sigma_n}]^p<\infty  ~~~{\rm implies}~~~   \limsup_n n^p\,\ex [X_{\sigma_n}]^p<\infty.
\end{equation}
In particular, $\lim\limits_{n\to\infty} n^p \,\ex [X_{\tau_n}]^p<\infty$ for $\tau_n$ the stopping time minimising $\ex [R_\tau]^p$
over all stopping times adapted to $X_1,\dots,X_n$.
\end{proposition}
\noindent The idea is to bound $X_{\sigma_n}$ by exploiting
properties of a random walk with negative drift.

\paragraph{2.} Let $\xi, \xi_1,\xi_2,\dots$ be iid nonnegative random variables with $\mu={\mathbb E}\xi \in (0,\infty)$.
Let $S_k:=\xi_1+\cdots+\xi_k$ and  for $\lambda>\mu$ let
$M_\lambda=:\sup_{k\geq 0}(S_k-\lambda k)$.

\begin{proposition}\label{main} For $p>0$
$${\mathbb E}\xi^{p+1}<\infty ~~~\Longleftrightarrow~~~ {\mathbb E} M_\lambda^p     <\infty.$$
\end{proposition}
\begin{proof}
The moment condition on $\xi$ is equivalent to ${\mathbb
E}[(\xi-\lambda)^+]^{p+1}<\infty$, and the result follows from
Lemma 3.5 in \cite{AlsIks}.
\end{proof}

\begin{corollary}\label{cor} Suppose $\ex \xi^{p+1}<\infty$ and
let $\sigma$ be a nonnegative integer random variable with
$\ex\sigma^p<\infty$. Then $\ex S_\sigma^p<\infty$.
\end{corollary}
\begin{proof} This follows from
$S_\sigma^p\leq (M_\lambda+\lambda\sigma)^p\leq c_p(M_\lambda^p+\lambda^p\sigma^p),$
where $c_p:=2^{p-1}\vee 1$.
\end{proof}

\paragraph{3.} We  can apply Corollary \ref{cor} to a Poisson-embedded, limiting form of the stopping problem with continuous time \cite{GnRank}.
Let $\xi_1,\xi_2,\dots$ be iid rate-one exponential variables,
$S_k$ as above, and let $T_1,T_2,\dots$ be iid uniform $[0,1]$
random times, independent of the $\xi_j$'s. The points $(T_k,S_k)$
are the atoms of a homogeneous planar Poisson process $\cal P$ in
$[0,1]\times[0,\infty)$. To introduce the dynamics, consider an 
observer whose information at time $t\in[0,1]$ is the (infinite)
configuration of points of $\cal P$ within the strip
$[0,t]\times[0,\infty)$, that is $\{(T_k,S_k): T_k\leq t\}$. The
rank of point $(T_k,S_k)$ is defined as $R_{T_k}=k$, meaning that
$S_k$ is the $k$th smallest value among $S_1, S_2,\dots$. The
piece of information added at time $T_k$ is the point $(T_k,S_k)$,
but not the  rank $R_{T_k}$.

Suppose the objective of the observer is to minimize $\ex [R_\tau]^p$ over stopping times $\tau$
that assume values in the random set $\{T_1,T_2\dots\}$
and are adapted to the information flow
of the observer. 
For the optimal stopping time $\tau_\infty$ it is known from the previous studies  that $\ex [R_{\tau_\infty}]^p<\infty$ (see \cite{GnRank} and \cite{Gianini}).
Taking $\sigma=R_{\tau_\infty}$, we have
$\ex [S_\sigma]^p<\infty$.
The case $p=1$ corresponds
to the infinite version of Robbins's problem of minimising the expected rank.

\paragraph{4.} To apply the above to a finite sample, we shall use the familiar realisation
of uniform order statistics through sums of exponential variables, as
$$(Y_k, ~1\leq k\leq n)\stackrel{d}{=}(S_k/S_n,~1\leq k\leq n).$$
%which is coupled with the sequence $S_1,S_2,\dots$ consistently for all $n$.
Introducing the event $A_n:=\{n/S_n> 1+\epsilon\}$, we can
estimate for $1\leq k\leq n$
$$n^pY_k^p=n^pY_k^p 1_{A_n}+n^pY_k^p 1_{A_n^c}\leq  n^p 1_{A_n}+(1+\epsilon)^\rho S_k^p \leq n^p 1_{A_n}+
c_p(1+\epsilon)^p(M_\lambda^p+\lambda^p k^p),$$ where we used
$S_k\leq M_\lambda +\lambda k$. Using a large deviation bound for
the probability of $A_n$ and sending $\epsilon\to 0$ we conclude
that for any random variable $\sigma_n$ with values in
$\{1,\ldots,n\}$
$$\limsup\limits_n n^p\ex [Y_{\sigma_n}]^p\leq c_p\lambda^p \limsup\limits_n \ex \sigma_n^p+c_p\ex M_\lambda^p.$$
Finally, taking $\sigma_n=R_{\tau_n}$, Proposition \ref{main}
follows from
$$ \limsup\limits_n n^p\ex [X_{\tau_n}]^p\leq c_p\lambda^p\limsup\limits_n \ex [R_{\tau_n}]^p+c_p\ex M_\lambda^p<\infty,$$
since $\ex [R_{\tau_n}]^p$ converges to a finite limit (see
\cite{GnRank}, \cite{Gianini}).

\vskip0.2cm
\noindent
{\bf Acknowledgement} This note was completed during the second author's visit to Utrecht, supported by the
Department of Mathematics and stochastic cluster STAR.


\begin{thebibliography}{99}

\bibitem{AlsIks} {\sc Alsmeyer, G. and Iksanov, A.} (2009). A
log-type moment result for perpetuities and its application to
martingales in supercritical branching random walks. {\em Elect.
J. Probab.} {\bf 14}, 289--313.

\bibitem{Assaf} {\sc Assaf, D. and Samuel-Cahn, E.} (1996). The secretary problem: minimizing the expected rank with i.i.d.
random variables. {\em Adv. Appl. Prob.} {\bf 28}, 828--852.

\bibitem{BrSurv} {\sc Bruss, F. T.} (2005). What is known about Robbins' problem? {\em J. Appl. Prob.} {\bf 42}, 108--120.

\bibitem{BrSw}{\sc Bruss, F.T. and Swan, Y.} (2009). A continuous time approach to Robbins' problem of minimizing the expected rank.
{\em J. Appl. Prob.} {\bf 46},  1-18.

\bibitem{Gianini}{\sc Gianini, J. and Samuels, S.M.} (1976). The infinite secretary problem.
{\em Ann. Probab.} {\bf 3}, 418--432.



\bibitem{GnRank} {\sc Gnedin, A.}  (2007).
 Optimal stopping with rank-dependent loss,
{\em J. Appl. Prob.} {\bf 44}, 996--1011.

%\bibitem{Kyprianou} Kyprianou, A. (2006). {\it Introductory Lectures on Fluctuations of L{\'e}vy processes with
%Applications}, Springer, Berlin.

\end{thebibliography}
\end{document}